\documentclass[11pt]{article}

\usepackage{graphicx,,psfrag}

\usepackage{amssymb,amsmath,amsthm,enumerate}
\usepackage{fullpage}
\newcommand{\url}{}

\newcommand{\dE}{\mathbb{E}}
\newcommand{\dP}{\mathbb{P}}
\newcommand{\dZ}{\mathbb{Z}}
\newcommand{\dN}{\mathbb{N}}
\newcommand{\dR}{\mathbb{R}}
\newcommand{\dC}{\mathbb{C}}

\newcommand{\ev}{\dE}
\newcommand{\pr}{\dP}
\newcommand{\one}{{\IND}}
\newcommand{\veps}{\varepsilon}

\newcommand{\cM}{\mathcal{M}}
\newcommand{\cG}{\mathcal{G}}
\newcommand{\cP}{\mathcal{P}}

\newcommand{\cB}{\mathcal{B}}

\newcommand{\POI}{\mathrm{Poi}}

\newcommand{\wP}{\widehat{P}}

\newcommand{\IND}{\mathbf{1}}
\renewcommand{\deg}{\mathrm{deg}}

\newcommand{\UNI}{\mathrm{uni}}

\newcommand{\RANK}{\mathrm{rank}}
\newcommand{\DIM}{\mathrm{dim}}
\newcommand{\codim}{\operatorname{codim}}

\newcommand\bp[1]{\noindent {\em Proof{#1}.} 
}
\def\ep{\hfill $\Box$}

\newtheorem{theorem}{Theorem}[section]
\newtheorem{definition}[theorem]{Definition}
\newtheorem{lemma}[theorem]{Lemma}
\newtheorem{proposition}[theorem]{Proposition}
\newtheorem{op}[theorem]{Question}
\newtheorem{corollary}[theorem]{Corollary}
\newtheorem{remark}[theorem]{Remark}

\title{Mean quantum percolation}
\author{Charles Bordenave\thanks{Research partially supported by ANR-11-JS02-005-01}
\and Arnab Sen \and B\'alint Vir\'ag\thanks{Research partially supported by the Canada Research Chair program and the NSERC
Discovery Accelerator Supplement}}

\begin{document}
\maketitle

\begin{abstract}
We study  the spectrum of adjacency matrices of random graphs. We develop two techniques to lower bound the mass of the continuous part of the spectral measure or the density of states. As an application, we prove that the spectral measure of bond percolation in the two dimensional lattice contains a non-trivial continuous part in the supercritical regime. The same result holds for the limiting spectral measure of a supercritical  Erd\H{o}s-R\'enyi graph and for the spectral measure of a unimodular random tree with at least two ends. We give examples of random graphs with purely continuous spectrum.
\end{abstract}

\section{Introduction}

This work is devoted to the spectral analysis of adjacency matrices of deterministic and random graphs (the latter is traditionally called quantum percolation).  The motivation comes from three distinct directions: random matrices, random trees, and random Schr\"odinger operators.

\subsection{Random matrices and Erd\H os-R\'enyi graphs}

Wigner introduced the study of random matrices to mathematical physics, and his first paper on the subject was on the density of states. He showed that
the empirical distribution of the eigenvalues of a random symmetric matrices with entries picked from some fixed distribution with exponential moments, converges, after scaling, to the famous Wigner semicircle law.

A particular example is the adjacency matrix of Erd\H os-R\'enyi random graphs on $n$ vetices, where each edge
is present with a fixed probability $p$.  In this case, it is not hard to show that the semicircle limit holds as $np\to\infty$. At the forefront of current research is the  sparse case, when $np\to c\in(0,\infty)$. The limit is quite unlike the Wigner semicircle law: it is supported on the entire real line and has a dense set of atoms, see \cite{MR869300}.
This and related questions have been discussed in the physics literature, see notably Bauer and Golinelli \cite{bauer01,bauerg01}.
When $c\to \infty$, the measure $\mu_c$ rescaled by $c$ converges to the Wigner semicircle law. This suggests, (but by no means implies) that there may be a continuous part for large $c$, which has been an open problem.

Our first theorem establishes this in a strong sense.
\begin{theorem}
The measure $\mu_c$ has a continuous part if and only if $c>1$.
\end{theorem}

This result is a corollary to our investigation of random trees.

\subsection{Percolation on regular trees}

The spectra of random trees is a very active area of random Schr\"odinger operators: it is in this setting
that the celebrated extended states conjecture was first proved: for a continuous perturbation of the regular tree.

A slightly different perturbation is given by Bernoulli bond percolation with high enough parameter $p$.
Here each edge of the graph of the regular tree $T_d$ is removed independently with probability $1 - p \in [0,1]$.

First, a quick definition: the spectral measure of the adjacency matrix of a (possibly infinite) bounded degree rooted graph is the unique probability measure
whose $k$th moments are given by the number of paths of length $k$. This can be extended to a more general setting, see Section \ref{ss:spectral measure}.

For the spectral measure (without taking expectation), for $p$ close to $1$, Keller \cite{Kelleretal} shows that the component of the root has continuous spectrum with positive probability. It is an open problem
to show when the continuous part appears. However, for expected spectral measure $\mu$, we can show

\begin{theorem} The
critical value for the existence of continuous part in the expected spectral measure  $\mu_c$ for percolation on $d$-regular trees is $1/d$.
\end{theorem}

This can be generalized to Galton-Watson trees. In fact, we will show that a bounded degree unimodular tree has continuous spectrum as long as
it has at least two ends, see Section \ref{ss:unimodular trees}.

In contrast  Bhamidi, Evans and Sen \cite{MR2956206} prove the limiting spectral measures for various popular models of  random trees, for example uniform random trees and trees generated by preferential attachment scheme,  have a dense set of atoms (this does not rule out the existence of some continuous part).
In \cite{MR2789584}, Lelarge, Salez and the first author show examples of Galton-Watson trees with arbitrary high minimal degree and  an atom at $0$ for the spectral measure.


\subsection{Percolation on Euclidean lattices}

The study of the regularity of the density of states is of prime importance in the literature on random Schr\"odinger  operators.
 The study of random Hamiltonians generated by
percolation on $\mathbb Z^d$ was initiated by De Gennes, Lafore and
Millot \cite{deGennes1959a,deGennes1959b} in the 1950's under
the name of quantum percolation.
The study of the density of states is a preliminary step toward into understanding the
behaviour of the eigenvectors, see e.g. Kirkpatrick and Eggarter \cite{PhysRevB.6.3598}, Chayes et al.\ \cite{MR869300} or Veseli{\'c} \cite{MR2148799}
and references therein. After more than a half-century, it is still  a very active field of research and proving the existence of Anderson delocalization remains the main open challenge in the area. One of the issue of quantum
percolation models is that the lack of regularity of percolation graphs does not allow to use Wegner estimates.

In parallel, the study of the spectral properties of graphs and countable groups has a long history, see Mohar and Woess \cite{MR986363} for an early survey on the matter. In \cite{MR0109367}, Kesten computed the spectral measure of the infinite $d$-regular tree (the Cayley graph of the free group with $d$-generators). This spectral measure is absolutely continuous. It is not always the case, in \cite{MR1866850}, Grigorchuk and {\.Z}uk have proved that the spectral measure of the usual lamplighter group is purely atomic, see also Lehner, Neuhauser and Woess \cite{MR2415315}.
Hence, neither connectivity nor regularity are necessary to guarantee the regularity of the spectral measure.

For site or bound percolation on $\mathbb Z^d$ the expected spectral measure $\mu$ can be defined through moments, spectral theory, or simply as the limit of the empirical eigenvalue distribution on finite boxes, see
Section \ref{ss:spectral measure}.

When the percolation has only finite components, the expected spectral measure is $\mu$ a countable mixture of atomic measures, so it is purely atomic. On the other end, for $p=1$ the percolation is simply $\dZ^d$ rooted at the origin and its spectral measure is absolutely continuous. In fact, it is the convolution of $d$ arcsine distributions, see \cite[Section 7.B]{MR986363}.

\begin{theorem} \label{thm:supercric_perc_Z2}
For bond percolation on $\mathbb Z^2$, the expected spectral measure has a continuous part if and only if $p>p_c$.
\end{theorem}

In this paper, the focus is on the adjacency operator of a graph. The same study could be generalized to discrete Laplacian or combinatorial Laplacian or weighted graphs. In the next section, we discuss unimodular random graphs, a convenient framework in which allow us to discuss all the above spectral  questions. Our contribution for this theory is that we can define the expected spectral measure for unimodular random graphs in complete generality, see Proposition \ref{prop:defspecmeas}.

\subsection{Unimodular random graphs}

We first briefly introduce the theory of local weak convergence of graph sequences and the notion of unimodularity. It was introduced by Benjamini and Schramm \cite{bensch} and has then become a popular topology for studying sparse graphs. Let us briefly introduce this topology, for further details we refer to Aldous and Lyons \cite{aldouslyons}.

A graph $G = (V,E)$ is said to be locally finite if for all $v \in V$, the degree of $v$ in $G$, $\deg_G(v)$, is finite. A {\em rooted graph} $(G,o)$ is a locally finite and connected graph $G = (V,E)$
with a distinguished vertex $o \in V$, the root. Two rooted graphs $(G_i,o_i) =  ( V_i , E_i , o_i )$,  $i \in \{1,2\}$, are \emph{isomorphic} if there exists a bijection $\sigma : V_{1} \to V_{2}$ such that $\sigma ( o_1) = o_2$ and $\sigma ( G_1) = G_2$, where $\sigma$ acts on $E_1$ through $\sigma (  \{ u , v \}  ) = \{ \sigma ( u) , \sigma (v) \}$. We will denote this equivalence relation by $(G_1,o_1) \simeq (G_2, o_2)$. In graph theory terminology, an equivalence class of rooted graph is an {\em unlabeled rooted graph}. We denote by $\cG^*$ the set of unlabeled rooted graphs.

The {\em local topology} is the smallest topology such that for any $g \in \cG^*$ and integer $t \geq 1$, the $\cG^* \to \{0,1\}$ function $
f(G,o) = \IND ( (G,o)_t \simeq g ) $
is continuous, where $(G,o)_t$ is the induced rooted graph spanned by the vertices at graph distance at most $t$ from $o$.  This topology is metrizable and the space $\cG^*$ is separable and complete.

For a finite graph $G = (V,E)$ and $v\in V$, one writes $G(v)$ for the connected component of $G$ at $v$. One defines the probability measure $U(G)\in\cP(\cG^*)$ as the law of the equivalence class of the rooted graph
$(G(o),o)$ where the root $o$ is sampled uniformly on $V$. If $(G_n)_{n \geq 1},$ is a sequence of finite graphs, we shall say that $G_n$ has {\em local weak limit}  $\rho\in\cP(\cG^*)$ if $U(G_n) \to \rho$ weakly in $\cG^*$.  A measure $\rho\in\cP(\cG^*)$ is called {\em sofic} if  there exists a sequence of finite graphs $(G_n)_{n \geq 1},$ whose local weak limit is $\rho$.

The notion of unimodularity can be thought of as invariance under moving the root, but it requires some subtlety to get the definition right.
Towards this end, we define locally finite connected graphs with two roots $(G,o,o')$ and extend the notion of isomorphisms to such structures. We define $\cG^{**}$ as the set of equivalence classes of  graphs $(G,o,o')$ with two roots and associate its natural local topology. A function $f$ on $\cG^{**}$ can be extended to a function on connected graphs with two roots $(G,o,o')$ through the isomorphism classes. Then, a measure $\rho\in\cP(\cG^*)$ is called {\em unimodular} if for any measurable function $f : \cG^{**} \to \dR_+$, we have
\begin{equation}\label{eq:defunimod}
\dE_\rho \sum_{ v \in V} f ( G , o , v) = \dE_\rho \sum_{ v \in V} f ( G , v , o),
\end{equation}
where under $\dP_\rho$, $(G,o)$ has law $\rho$. It is immediate to check that if $G$ is finite then $U(G)$ is unimodular. More generally, all sofic measures are unimodular, the converse is open, for a discussion see \cite{aldouslyons}. It is however known that all unimodular probability measures supported on rooted trees are sofic, see Elek \cite{elek2010}, Bowen \cite{bowen} and Benjamini, Lyons and Schramm \cite{BLS}. We will denote by $\cP_{\UNI} (\cG^*)$ the set of unimodular measures. It is closed under the local weak topology.

Any Cayley graph $G$ of a finitely generated group $\Gamma$ is unimodular (more precisely, for any $v \in \Gamma$, the measure $\rho$ which puts a Dirac mass at the equivalence class of $(G,v)$  is unimodular), see \cite[Section 3]{aldouslyons}.

With a slight abuse of language, we shall say that a random rooted $(G,o)$ is unimodular if the law of its equivalence class in $\cG^*$ is unimodular.

\subsection{The spectral measure of graphs}\label{ss:spectral measure}

Let $V$ be countable and $G = (V,E)$ be a locally finite graph, its {\em adjacency operator}, denoted by $A$, is defined for vectors $\psi \in \ell^ 2 (V)$ with finite support by  the formula
$$
A \psi (u) = \sum_{v : \{ u, v\} \in E } \psi(v).
$$
By construction $A$ is symmetric. Assume furteher that the degrees of vertices are bounded by an integer $d$, then we readily check that $A$ has norm bounded by $d$. Hence, $A$ is a self-adjoint operator. For any $\psi \in \ell^2 (V)$ with $\| \psi \|^ 2_2 = 1$, we may thus define the {\em spectral measure with vector} $\psi$, denoted by $\mu^\psi_A$, as the unique probability measure on $\dR$, such that for all integers $k\geq 1$,
$$
 \int x^ k d \mu^ \psi _ A = \langle \psi , A^ k \psi \rangle.
$$
For example if $|V| = n$ is finite, then $A$ is a symmetric matrix. If $(v_1, \cdots, v_n)$ is an orthonormal basis of eigenvectors associated to eigenvalues $(\lambda_1, \cdots, \lambda_n)$, we find
\begin{equation}\label{eq:defspecmeas}
\mu^ \psi _ A = \sum_{k =1} ^  n \langle v_k , \psi \rangle ^ 2 \delta_{\lambda_k}.
\end{equation}
(In $V$ is not finite, $\mu^ \psi_A$ has a similar decomposition over the resolution of the identity of $A$ but we shall not need this here).

Now, we denote by $e_v (u) = \IND_{\{u = v\}}$, the coordinate vector associated to $v\in V$.  Remark that if two rooted graphs are isomorphic then the spectral measures associated to the coordinate vector for the root (simply called the spectral measure at the root) are equal. It thus makes sense to define $\mu_A^ {e_o}$ for elements of $\cG^ *$. Then, if $\rho \in \cP ( \cG^*)$ is supported on graphs with bounded degrees, we may consider the expected spectral measure of the root :
\begin{equation}\label{eq:defmurho}
\mu_\rho = \dE_\rho \mu^{e_o}_ A.
\end{equation}
In particular, if $|V| = n$ is finite, \eqref{eq:defspecmeas} implies
$$
\mu_{U(G)} = \frac{1}{n} \sum_{k=1} ^ n  \delta_{\lambda_k}.
$$
It is the {\em empirical distribution of the eigenvalues} of the adjacency matrix.

It is not clear a priori how to extend this construction to random graphs without bounded degrees. It can be difficult to check that adjacency operators are essentially self-adjoint (for a criterion of essential self-adjointness of the adjacency operator of trees, see \cite{MR2789584} and for a characterization see Salez \cite[theorem 2.2]{salezphd}).  It turns out however that for unimodular measures, it is always possible to define $\mu_\rho$ without any bounded degree assumption.

\begin{proposition}\label{prop:defspecmeas}
For any $\rho \in \cP_{\UNI}( \cG^ *)$, there exists a unique $\mu_\rho \in \cP( \dR)$ such that 
\begin{enumerate}[(i)]
\item
if the adjacency operator $A$ is $\rho$-a.s.\ essentially self adjoint, then $\mu_\rho$ is given by \eqref{eq:defmurho}. 
\item
if $\rho_n \in \cP_{\UNI}( \cG^ *)$ and $\rho_n \to \rho$, then $\mu_{\rho_n}$ converges weakly to $\mu_\rho$.  
\end{enumerate}
\end{proposition}

Proposition \ref{prop:defspecmeas} is equivalent to the following: there is a unique continuous extension of the map $\rho \to \mu_\rho$ defined on the dense set of bounded degree graphs via \eqref{eq:defmurho}. 

If a sequence of finite graphs $(G_n)_{n \ge 1}$ has local weak limit $\rho$ then the empirical distribution of the eigenvalues of their adjacency matrices converges weakly to $\mu_\rho$. In this last case, if moreover for some $\theta >0$ and for all $v \in V(G_n)$, $\deg_{G_n} (v) \leq \theta$, then using L\"uck's approximation (refer to \cite{MR1926649, thom08,ATV}), the convergence can even be reinforced to the pointwise convergence of all atoms\footnote{This uniformly bounded degree assumption can also be lifted by using the truncation argument used in the proof of Proposition \ref{prop:defspecmeas}. We will however not need this refinement.}.

\subsection{Existence of continuous spectrum in unimodular trees}\label{ss:unimodular trees}

In this paper, we will develop two simple tools to prove the existence of a continuous part of the spectral measure of unimodular graphs. In addition
to the examples given in the beginning of the Introduction, will present many cases where these two tools can be applied.

A weighted graph $(G,\omega)$ is a graph $G = ( V,E)$ equipped with a weight function $\omega : V^2 \to \dZ$ such that $\omega (u,v) = 0$ if $u \ne v$ and $\{u,v\} \notin E$. The weight function is edge-symmetric if $\omega(u,v) = \omega(v,u)$ and $\omega (u,u) = 0$. Note that, for edge-symmetric weight functions, the set of edges such that $\omega (e) = k$ spans a subgraph of $G$. A {\em line ensemble} of $G$ is a edge-symmetric weight function $L : V^2 \to \{0,1\}$ such that for all $v \in V$,
$$
\sum_{u} L (u,v) \in \{0,2\}.
$$
We will think of $L$ as a subgraph of $G$ which consists of a union of vertex-disjoint copies of $\mathbb Z$.

It is straightforward to extend the local weak topology to weighted graphs. The definition of unimodularity carries over naturally  to the weighted graphs (see the definition of unimodular network in  \cite{aldouslyons}). Now, consider a unimodular graph $(G,o)$. If, on an enlarged probability space, the weighted graph $(G,L,o)$ is unimodular and $L$ is a.s.\ a line ensemble then we shall say that $L$ is {\em an invariant line ensemble} of $(G,o)$. We call $\pr(o\in L)$ the {\em density} of $L$.


 \begin{theorem}\label{th:treesILE}
Let $(T,o)$ be a unimodular tree with law $\rho$. If $L$ is an invariant line ensemble of $(T,o)$ then  for each real $\lambda$,
\[ \mu_\rho (\lambda)   \leq \dP ( o \notin L ) \mu_{\rho'} (\lambda)\]  where, if $\dP ( o \notin L) >0$, $\rho'$ is the law of the rooted tree $( T \backslash L , o) $ conditioned on the root $o \notin L$.  In particular, the total mass of atoms of $\mu_\rho$  is bounded above by $\dP ( o \notin L)$.
 \end{theorem}

We will check in \S \ref{subsec:stab} below that the measure $\rho'$ is indeed unimodular. As a consequence, if $(T,o)$ has an invariant line ensemble such that $\dP( o \in L) =1$ then $\mu_\rho$ is continuous.  Our next result gives the existence of invariant line ensemble for a large class of  unimodular trees. We recall that for a rooted tree $(T, o)$, a topological  end is just an infinite non-backtracking  path in $T$ starting from $o$.

\begin{proposition} \label{prop:ILE}
Let $(T,o)$ be a unimodular tree. If $T$ has at least two topological ends with positive probability, then $(T,o)$ has an invariant line ensemble $L$ with positive density: $\dP ( o \in L) >0$. Moreover, we have the following lower bounds.
\begin{enumerate}[(i)]
\item{} $\dP ( o \in L) \ge \frac{1}{6}\frac{(\ev\, \deg(o)-2)^2}{\ev\, \deg(o)^2}$ as long as the denominator is finite.
\item{} Let $q$ be the probability
that $T\setminus \{o\}$ has at most one infinite component.  If $\deg(o)\le d$ a.s., then $\dP ( o \in L) \ge \frac{1}{3} \,{(\ev\, \deg(o)-2q)}/{d}$.
\end{enumerate}
\end{proposition}

One of the natural example where the conditions of Proposition~\ref{prop:ILE} are not satisfied is the infinite skeleton tree which consists of a semi-infinite line $\mathbb Z_+$ with i.i.d.\ critical Poisson Galton-Watson trees attached to each of the vertices of $\dZ_+$.  It is the local weak limit of the uniform trees on $n$ labeled vertices.

Let $P \in \cP(\dZ_+)$ with positive and finite mean. The unimodular Galton-Watson tree with degree distribution $P$ (commonly known as size-biased Galton-Watson tree) is the law of the random rooted tree obtained as follows. The root has a number $d$ of children sampled according to $P$, and, given $d$, the subtrees of the children of the root are independent Galton-Watson trees with offspring distribution
\begin{equation}\label{eq:defwP}
\wP ( k ) =   \frac{ (k+1) P(k+1) }{ \sum_{\ell} \ell P (\ell)}.
\end{equation}
 These unimodular trees appear naturally as a.s.\ local weak limits of random graphs with a given degree distribution,  see e.g.\ \cite{MR2656427,MR2643563,notesRG}. It is also well known that the Erd\H{o}s-R\'enyi  $G(n,c/n)$ has a.s.\  local weak limit the Galton-Watson tree with offspring distribution  $\POI(c)$. Note that if $P$ is $\POI(c)$ then $\wP = P$. The percolation on the hypercube $\{0,1\}^ n$ with parameter $c/n$ has the same a.s.\  local weak limit.

If $P$ has first moment $\mu_1$ and second moment $\mu_2$, then the first moment of $\wP$ is $\widehat \mu = (\mu_2 - \mu_1 )/ \mu_1$. If $P  \ne \delta_2$ and $\widehat \mu \leq 1$, then the unimodular Galton-Watson tree is a.s. finite. If $\widehat \mu > 1$ ($\widehat \mu = \infty$ is allowed), the tree is infinite with positive probability. Proposition \ref{prop:ILE} now implies the following phase transition exists for the existence of a continuous part  in the spectral measure.

\begin{corollary}\label{cor:treesILE}
Let $\rho$ be a unimodular Galton-Watson tree with degree distribution $P \ne \delta_2$. The first moment of $\wP$ is denoted by $\widehat \mu$. Then $\mu_{\rho}$ contains a non-trivial continuous part if and only if $\widehat \mu > 1$.
\end{corollary}

Note that for some choices of $P$, it is false that the total mass of the atomic part of $\mu_\rho$ is equal to the probability of extinction of the tree, it is only a lower bound (see \cite{MR2789584}).

Let us conclude the intoduction with a few open questions.

\subsection{Open questions}

\begin{op}
Consider a unimodular Galton-Watson tree with degree distribution $P$ with finite support and $P(0)=P(1)=0$. Does the expected spectral measure have only finitely many atoms?
\end{op}

Theorem \ref{thm:supercric_perc_Z2} naturally inspires the following question. We strongly believe that the answer is yes.

\begin{op}
Does supercritical bond percolation on $\mathbb Z^d$ have a continuous part in its expected spectral measure for every $d \geq 2$?
\end{op}

In view  of the result of Grigorchuk and {\.Z}uk  \cite{MR1866850} on the lamplighter group, the next problem has some subtlety
\begin{op}
Is there some monotonicity  in the weights of the atoms of the spectral measure (for some non-trivial partial order on unimodular measures)?
\end{op}

Our main results concern percolation on lattices and trees. It motivates the following  question.
\begin{op}
What can be said about the regularity of the spectral measure for other nonamenable/hyperbolic graphs and for other planar graphs (such as the uniform infinite planar triangulation in Angel and Schramm \cite{MR2013797})?
\end{op}

We have seen that regular trees with degree at least 2 contain invariant line ensembles with density 1. A quantitative version of this would be that if the degree is concentrated, then the density is close to 1. Based on the last part of Proposition \ref{prop:ILE}.
the following formulation is natural.

\begin{op}
Is there a function $f$ with $f(x)\to 1$ as $x\to 1$ so that
every unimodular tree of maximal degree $d\ge 2$ contains and invariant line ensemble with density
at least $f(\ev \deg(o)/d)$?
\end{op}

Two open questions (Questions \ref{q:mi1} and \ref{q:mi2}) can be found in Section \ref{sec:trees}.

\section{The monotone labeling technique}
\label{subsec:tool2}

In this section we will use a carefully chosen labeling of the vertices of a graph to  prove regularity of  its spectrum. The intuition being that a labeling gives an order to solve the eigenvalue equation at each vertex.

\begin{definition}
Let $G = (V,E)$ be a graph. A map $\eta : V \to \dZ$ is a labeling of the vertices of $G$ with integers.  We shall call a vertex $v$
\begin{itemize}
\item[(i)] {\bf prodigy} if it has a neighbor $w$ with $\eta(w)<\eta(v)$ so that all other neighbors of $w$ also have label less than $\eta(v)$,
\item[(ii)] {\bf level} if not prodigy and if all of its neighbors have the same or lower labels,
\item[(iii)] {\bf bad} if none of the above holds.
\end{itemize}
\end{definition}

%

\paragraph{Finite graphs.} We start with the simpler case of finite graphs.

\begin{theorem}\label{th:tool2f}
Let $G$ be a finite graph, and consider a labeling $\eta$ of its vertices with integers.
Let $\ell,b$ denote the number of level and bad vertices, respectively. For any eigenvalue $\lambda$ with multiplicity $m$ we have, if $\ell_{j}$ is the multiplicity of the eigenvalue $\lambda$ in the subgraph induced by level vertices with label $j$,
$$
m \leq b + \sum_j \ell_{j}.
$$
Consequently, for any multiplicities $m_1,\ldots,m_k$ of distinct eigenvalues we have
$$
m_1+\ldots +m_k  \le kb +\ell.
$$
\end{theorem}

\bp{}
Let $S$ be the eigenspace for the eigenvalue $\lambda$ of multiplicity $m$. Consider
the set of bad vertices, and let $B$ be the space of vectors which vanish on that set.
For every integer $j$, let $L_j$ denote the set of level vertices with label $j$ and let $A_{j}$ denote the
eigenspace of $\lambda$ in the induced subgraph of $L_j$. With the notation of the theorem, $\DIM( A_j) = \ell_j$. We extend the vectors in $A_{j}$ to the whole graph by setting
them to zero outside $L_j$. Let $A_{j}^\perp$ be the orthocomplement of $A_{j}$.
Recall that for any vector spaces $A,B$ we have $\DIM ( A \cap B ) \ge \DIM A -\codim B$. Using this, let $S'=S\cap B \cap \bigcap_{j} A^{\perp}_{j}$,
and note that
\begin{equation}\label{eq:codimH}
\DIM S'\ge \DIM S - \codim B  - \sum_{j}  \codim  A_{j} ^\perp   = m - b - \sum_j \DIM A_{j}.
\end{equation}
However, we claim that the subspace $S'$ is trivial. Let $f\in S'$. We now prove, by induction on the label $j$ of the vertices, low to high, that $f$ vanishes on vertices with label $j$. Suppose that $f$ vanishes on all vertices with label strictly below  $j$. Clearly, $f$ vanishes on all bad vertices
since $f\in B$. Consider a prodigy $v$ with label $j$. Then, by induction hypothesis, $v$ has a neighbor $w$ so that $f$ vanishes on all of the neighbors of $w$
except perhaps at $v$. But the eigenvalue equation
$$
\lambda f(w)=\sum_{u\sim w} f(u)
$$
implies that $f$ also vanishes at $v$. Now, observe that the outer vertex boundary of $L_j$ (all vertices that
have a neighbor in $L_j$ but are not themselves in $L_j$) is contained in the union of the set of bad vertices, the set of level vertices with label strictly below $j$ and the set of prodigy with label $j$. Hence, we know that $f$ vanishes on the outer vertex boundary of $L_j$. This means that the restriction of $f$ to $L_j$ has to satisfy the eigenvector equation.
But since $f\in A^{\perp}_{j}$, we get that $f(v)=0$ for $v\in L_j$, and the induction is complete.

We thus have proved that $S'$ is trivial. Thus Equation \eqref{eq:codimH} implies that $m \le b + \sum_j \DIM A_{j}$. It gives the first statement of Theorem \ref{th:tool2f}.

For the second statement, let $A_{i,j}$ denote the
eigenspace of $\lambda_i$ in the induced subgraph of $L_j$. Summing over $i$ the above inequality, we get
\[
m_1+\ldots +m_k \le bk + \sum _{j}\sum_{i} \DIM A_{i,j}\le bk + \sum _{j}|L_j| = bk+ \ell.
 \]
\ep
\paragraph{Unimodular graphs.}

We now prove the same theorem for unimodular random graphs  which may possibly be infinite. To make the above proof strategy work, we need a suitable notion of normalized dimension for  infinite dimensional subspaces of $\ell^2(V)$.
This requires  some basic concepts of operator algebras. First, as usual, if $(G,o)$ is a unimodular random graph, we shall say that a labeling  $\eta : V(G) \to \dZ$ is  invariant if on an enlarged probability space, the vertex-weighted rooted graph $(G,\eta,o)$ is unimodular.

There is a natural Von Neumann algebra associated to unimodular measures. More precisely,   let $\cG^*$ denote the set of equivalence classes of locally finite connected (possibly weighted) graphs endowed with the local weak topology. There is a canonical way to represent an element $ (G,o) \in \cG^ *$ as a rooted graph on the vertex set $V(G) = \{ o , 1 , 2 , \cdots, N \}$ with root $o$ and $N \in \dN \cup \{ \infty\}$, see Aldous and Lyons \cite{aldouslyons}. We set $V = \{o, 1, 2,  \cdots\}$, $H = \ell^2 (V)$  and define $\cB(H)$ as the set of bounded linear operators on $H$. For any bijection $\sigma : V \to V$, we consider the orthogonal operator $\lambda_{\sigma}$ defined for all $v \in V$, $\lambda_{\sigma} (e_v) = e_{\sigma(u)}$. For a fixed $\rho \in \cP_{\UNI} (\cG^*)$, we introduce the algebra $\cM$ of operators in $ L^{\infty} ( \cG^*, \cB (H),\rho)$ which commutes with the operators $\lambda_{\sigma} $, i.e. for any bijection $\sigma$, $\rho$-a.s. $B(G,o) = \lambda^{-1}_{\sigma} B ( \sigma(G),o ) \lambda_{\sigma}$. In particular, $B(G,o)$ does not depend on the root. It is a von Neumann algebra and the linear map $\cM \to \dC$ defined by
$$\tau (B) = \dE_\rho \langle e_o , B e_o \rangle,$$
where $B = B(G,o) \in \cM$ and under, $\dE_\rho$, $G$ has distribution $\rho$, is a normalized trace (see \cite[\S 5]{aldouslyons} and Lyons \cite{MR2593624}). By construction, an element of $B \in \cM$ is a random bounded operator associated to the random rooted graph $G$.

A closed vector space $S$ of $H$ such that, $P_S$, the orthogonal projection to $S$, is an element of $\cM$ will be called an invariant subspace. Recall that the von Neumann dimension of such vector space $S$ is just
\[\DIM(S):= \tau ( P_S ) = \dE_{\rho} \langle e_o, P_S e_o \rangle.\]
We refer e.g.\ to Kadison and Ringrose \cite{MR1468230}.

\begin{theorem}\label{th:tool2}
Let $(G,o)$ be unimodular random graph with distribution $\rho$, and consider an invariant labeling $\eta$ of its vertices with integers.
Let $\ell,b$ denote the probability that the root is level or bad, respectively. For integer $j$ and real $\lambda$, let $\ell_{j}$ be the von Neumann dimension of the eigenspace of $\lambda$ in the subgraph spanned by level vertices with label $j$. The spectral measure $\mu_\rho$ satisfies
$$
\mu_{\rho} ( \lambda) \leq b + \sum_j \ell_j.
$$
Consequently, for any distinct real numbers  $\lambda_1,\ldots,\lambda_k$,  we have
$$
\mu_{\rho}(\lambda_1)+\ldots +\mu_{\rho}(\lambda_k)  \le kb +\ell.
$$
In particular, if $b=0$, then the atomic part of $\mu_{\rho}$ has total weight at most $\ell$.
\end{theorem}
\bp{}
We first assume that there are only finitely many labels. Let $S$ be the eigenspace of $\lambda$ : that is the subspace of $f\in  \ell^2 (V)$ satisfying, for all $w \in V$,
\begin{equation}\label{eeq2}
\lambda f(w)=\sum_{u\sim w} f(u).
\end{equation}
Consider the set of bad vertices, and let $B$ be the space of vectors which vanish on that set.
For every integer $j$ let $L_j$ denote the set of level vertices with label $j$. Let $A_{j}$ denote the
eigenspace of $\lambda$ in the induced subgraph of $L_j$; extend the vectors in $A_{j}$ to the whole graph by setting
them to zero outside $L_j$. Let $A_{j}^\perp$ be the orthocomplement of $A_{j}$.

For any two invariant vector spaces $R$,$Q$ we have
$$\DIM(R \cap Q) \ge \DIM(R) +\DIM (Q) -1,$$
(see e.g. \cite[exercice 8.7.31]{MR1170351}). Setting $S'=S\cap B \cap \bigcap_{j} A^{\perp}_{j}$,
it yields to
$$
\DIM (S')\ge \DIM (S) +\DIM (B)- 1  + \sum_j (\DIM (A^{\perp}_{j}) -1)= \mu_\rho(\lambda_i) - b - \sum_j \DIM ( A_{j}).
$$
However, we claim that the subspace $V_i'$ is trivial. Let $f\in V_i'$. We now prove, by induction on the label $j$ of the vertices, low to high, that $f$ vanishes on vertices with label $j$. The argument is exactly similar to the case of  finite graphs presented before.  Suppose that $f$ vanishes on all vertices with label strictly below  $j$. Clearly, $f$ vanishes on all bad vertices
since $f\in B$. Consider a prodigy $v$ with label $j$. Then $v$ has a neighbor $w$ so that $f$ vanishes on all of the neighbors of $w$
except perhaps at $v$. But the eigenvalue equation \eqref{eeq2}
implies that $f$ also vanishes at $v$. By now, we know that $f$ vanishes on the outer vertex boundary of $L_j$. This means that the restriction of $f$ to $L_j$ has to satisfy the eigenvector equation.
But since $f\in A^{\perp}_{j}$, we get that $f(v)=0$ for $v\in L_j$, and the induction is complete.

We have proved that $\mu_\rho(\lambda_i)\le b + \sum_j \DIM (A_{j})$ : it is the first statement of the theorem  in the case of finitely many labels.  When there are infinitely many labels, for every $\veps$, we can find $n$ so that $\pr(|\eta(o)|>n)\le \veps$. We can relabel all vertices with $|\eta(v)|>n$ by $-n-1$; this may make them bad vertices, but will not make designation of vertices with other labels worse. The argument for finitely many labels gives
\[
\mu_\rho(\lambda) \le b+\veps+ \sum_{j = -n - 1} ^n \dim ( A_j) \leq b + 2\veps + \sum_{j = -n} ^n \dim(A_j) \leq b + 2 \veps + \sum_j \ell_j,
\]
and letting $\veps\to 0$ completes the proof of the first statement.

For the second statement, let $A_{i,j}$ denote the eigenspace of $\lambda_i$ in the induced subgraph of $L_j$. Summing over $i$ the above inequality, we get
\[
\mu_\rho(\lambda_1)+\ldots +\mu_\rho(\lambda_k) \le bk + \sum _{j}\sum_{i} \DIM (A_{i,j}) \le bk + \sum _{j}\pr (o\in L_j) = bk+ \ell.
\]


\ep

\paragraph{Vertical percolation.} There are simple examples where we can apply Theorems \ref{th:tool2f}-\ref{th:tool2}. Consider the graph of $\dZ^2$. We perform a vertical percolation by removing some vertical edge $\{ (x,y), (x,y+1) \}$. We restrict to the $n \times n$ box $[0,n-1]^2 \cap \dZ^2$. We obtain this way a finite graph $\Lambda_n$ on $n^2$ vertices.  We consider the labeling $\eta ( (x,y) ) = x$. It appears that all vertices with label different from $0$ are prodigy.  The vertices on the $y$-axis are bad and there are no level vertices. By Theorem \ref{th:tool2f}, the multiplicity of any eigenvalue of the adjacency matrix of $\Lambda_n$ is bounded by $n= o  (n^2)$.

Similarly, let $p \in [0,1]$. We remove each vertical edge $\{ (x,y), (x,y+1) \}$ independently with probability  $1-p$. We obtain a random graph $\Lambda(p)$ with vertex set $\dZ^2$. Now, we root this graph $\Lambda(p)$ at the origin and obtain a unimodular random graph. We claim that its expected spectral measure $\mu_\rho$ is continuous for any $p \in [0,1]$. Indeed, let $k \geq 1$ be an integer and $U$ be a random variable sampled uniformly on $\{0,\cdots, k-1\}$. We consider the labeling $\eta ( (x,y) ) = x + U \;  \mathrm{mod} (n)$. It is not hard to check that this labeling is invariant. Moreover, all vertices such that $\eta ( x , y ) \ne 0$ are prodigy while vertices such that $\eta (x,y) = 0$ are bad. It follows from Theorem \ref{th:tool2} that the mass of any atom of $\mu_\rho$ is bounded by $1/k$. Since $k$ is arbitrary, we deduce that $\mu_\rho$ is continuous.

The same holds on $\dZ^d$, $d \geq 2$, in the percolation model where we remove edges of the form $\{ u ,  u + e_k \}$, with $u \in \dZ^d$, $k \in \{2,\cdots, d\}$.

\section{The minimal path matching technique}

In this section, we give a new tool to upper bound the multiplicities of eigenvalues.
\begin{definition}
Let $G = (V,E)$ be a finite graph, $I = \{i_1, \cdots, i_b\}$ and $J =\{j_1, \cdots, j_b \}$ be two disjoint subsets of $V$ of equal cardinal. A {\bf path matching} $\Pi = \{\pi_\ell\}_{1\leq \ell \leq b}$ from $I$ to $J$ is a collection of self-avoiding paths $\pi_\ell = ( u_{\ell,1} , \cdots , u_{\ell, p_\ell} )$ in $G$ such that  for some permutation $\sigma$ on $\{1,\cdots, b\}$ and all  $1 \leq \ell \ne \ell' \leq b$,
\begin{itemize}
\item $\pi_{\ell'} \cap \pi_{\ell} = \emptyset$,
\item $u_{\ell,1} = i_{\ell}$ and $u_{\ell, p_{\ell}} =  j_{\sigma (i_\ell)}$.
\end{itemize}
We will call $\sigma$ the {\bf matching map} of $\Pi$. The length of $\Pi$ is defined as the sum of the  lengths of the paths
$$
|\Pi| = \sum_{\ell = 1} ^b | \pi_{\ell} | =  \sum_{\ell = 1} ^b | p_{\ell}|.
$$
Finally, $\Pi$ is a  {\bf minimal path matching} from $I$ to $J$ if its length is minimal among all possible paths matchings.
\end{definition}

Connections between multiplicities of eigenvalues and paths have already been known for a long time, see e.g.\ Godsil \cite{MR791025}. Kim and Shader \cite[Theorem 8]{MR2470115} provide a nice argument that connects the two notions in trees. This was the starting point for the proof of the following theorem.

\begin{theorem}\label{th:tool1}
Let $G = (V,E)$ be a finite graph and $I , J \subset V$ be two subsets of  cardinal $b$. Assume that the sets of path matchings from $I$ to $J$ is not empty and that all minimal path matchings from $I$ to $J$ have the same matching map. Then if $|V| - \ell$ is the length of a minimal path matching and if $m_1, \cdots, m_r$ are the multiplicities of the distinct eigenvalues $\lambda_1, \cdots, \lambda_r$ of the adjacency matrix of $G$, we have
$$
\sum_{i=1}^r ( m_i - b )_+ \leq  \ell.
$$
Consequently, for any $1 \leq k \leq r$,
$$
m_1 + \cdots + m_k  \leq k b + \ell.
$$
\end{theorem}

We will aim at applying Theorem~\ref{th:tool1} with $b$ small and $|V| - \ell$ proportional to $|V|$. Observe that $\ell$ is the number of vertices not covered  by the union of paths involved in a minimal path matching.  Theorem~\ref{th:tool1} and Theorem~\ref{th:tool2} have the same flavor but they are not equivalent one from each other. We note  that, contrary to Theorem \ref{th:tool2f}-Theorem \ref{th:tool2},  we do not have a version of Theorem \ref{th:tool1}  which holds for possibly infinite unimodular graphs. Unlike Theorem \ref{th:tool2f}, we do not have either a version which bounds the multiplicity of an eigenvalue in terms of its multiplicities in subgraphs. On the other hand, Theorem~\ref{th:tool1} will be used to show  the existence of non-trivial continuous part for the expected spectral measure of  two dimensional supercritical bond percolation.  It is not clear how to apply Theorem~\ref{th:tool2f} or Theorem~\ref{th:tool2} to get this result.

Following \cite{MR2470115}, the proof of Theorem \ref{th:tool1} is based on the divisibility properties of characteristic polynomials of subgraphs. For $I,J \subset V$, we define $(A - x)_{I,J}$ has the matrix $(A-x)$ where the rows with indices in $I$ and  columns with indices in  $J$ have been removed. We define the polynomial associated to the $(I,J)$-minor as :
$$
P_{I,J} (A) : x \mapsto  \det ( A - x )_{I,J}.
$$
We introduce the polynomial
$$
\Delta_{b} ( A) = \mathrm{GCD} \left( P_{I,J}(A)  : |I| = |J| = b \right),
$$
where $\mathrm{GCD}$ is the (unique) monic polynomial $g$ of highest degree so that all arguments are some polynomial multiple of $g$. Recall also that any polynomial divides $0$.  Observe that if $|I| = b$ then $P_{I,I}(A)$ is a polynomial of degree $|V|-b$. It follows that the degree of $\Delta_b$ is at most $|V| - b$.

The next lemma is the key to relate multiplicities of eigenvalues and characteristic polynomial of subgraphs.

\begin{lemma}\label{le:detminor}
If $A$ is the adjacency matrix of a finite graph and $m_1, \cdots, m_r$ are the multiplicities of its distinct eigenvalues $\lambda_1, \cdots, \lambda_r$, we have
$$
\Delta_b (A) = \prod_{i=1}^r ( x - \lambda_i )^{(m_i - b)_+}.
$$
 Consequently,
$$
\sum_{i=1}^r (m_i - b)_+ = \deg (\Delta_b (A)).
$$
\end{lemma}

\bp{}
We set $|V| = n$. If $B(x) \in \cM_n (\dR[x])$ is an $n \times n$   matrix with polynomial entries, we may define analogously $P_{I,J} ( B(x)) = \det  B(x)_{I,J} $ and $\Delta_b (B(x))$ (we retrieve our previous definition with $B(x) = A - x$).  Let $B_1(x), \cdots, B_n(x)$ be the columns of $B(x)$. The multi-linearity of the determinant implies that
\begin{align*}
& \det ( w_{11}  B_1 (x) + w_{21} B_2(x) + \cdots + w_{n1} B_n (x) , B_2(x) , \cdots, B_n(x) )_{I,J}  \\
&\quad  = \; \sum_{j=1} ^n w_{j1} \det ( B_j(x) , \cdots, B_n(x) )_{I,J^{(j)}}
\end{align*}
is a weighted sum of determinants of the minors  of the form  $(I, J^{(j)})$, where $J^{(j)} = (J\setminus \{1\}) \cup \{j\} $ if $1 \in J$ and $J^{(j)} = J$ if   $1 \notin J.$
 It is thus divided by $\Delta_b(B(x))$. The same holds for the rows of $B(x)$. We deduce that if $U,W \in \cM_n (\dR)$, $\Delta_b (B(x))$ divides $
\Delta_b ( U B(x) W ) $.  It follows that if $U$ and $W$ are invertible
$$
\Delta_b ( U B(x) W ) = \Delta_b (  B(x) ).
$$

We may now come back to our matrix $A$. Since $A$ is symmetric, the spectral theorem gives $A=  U D U^* $ with $U$ orthogonal matrix and $D$
diagonal matrix  with $m_i$ entries equal to $\lambda_i$.  We have $U ( D - x) U^* = A -x$. Hence, from what precedes
$$
\Delta_b (A - x) = \Delta_b (  D - x ).
$$
It is immediate to check that if $I \ne J$, $P_{I,J} ( D -x) = 0$ and
$$
P_{I,I} ( D - x) = \prod_{k \notin I} ( D_{kk} - x)  = \prod_{i=1} ^r ( \lambda_i - x)^{m_i - m_i ( I)},
$$
where $m_i (I) = \sum_{k \in I} \IND (D_{kk} =  \lambda_i ) $.  The lemma follows easily.
\ep

\vline

\bp{ of Theorem \ref{th:tool1}}
We set $|V| = n$. We can assume without loss of generality that  the matching map of minimal length matchings is the identity. We consider the matrix $B \in \cM_n ( \dR)$ obtained from $A$ by setting
$$
\hbox{for $1 \leq \ell \leq b$, } \, B e_{j_\ell}= e_{i_{\ell}} \; \hbox{ and for $j \notin J$, } \, B e_j = \sum_{i \notin I} A_{ij} e_i.
$$
In graphical terms, $B$ is the adjacency matrix of the oriented graph $\bar G$ obtained from $G$ as follows : (1)  all edges adjacent to a vertex in $J$ are oriented inwards,
 (2)  all edges adjacent to a vertex in $I$ are oriented outwards, (3) all other edges of $G$  have both orientations,  and (4) for each $1 \leq \ell \leq b$, an oriented edge from $j_{\ell}$ to $i_{\ell}$  is added. We define
$$
B(x) = B - x D,
$$
where $D$ is the diagonal matrix with entry $D_{ii} = 1 - \IND ( i \in I \cup J)$. Expanding the determinant along the columns $J$, it is immediate to check that
$$
\det B(x) = \det ( A -x)_{I,J}.
$$
We find
$$
P_{I,J} (A) = \sum_{\tau} (-1)^\tau \prod_{v \in V} B(x)_{v ,\tau(v)} = \sum_{\tau} (-1)^\tau Q_\tau(x),
$$
where the sum is over all permutations of $V$. Consider a permutation such that $Q_{\tau} \ne 0$. We decompose $\tau$ into disjoint cycles. Observe that $Q_\tau \ne 0$ implies that any cycle of length at least $2$ coincides with a cycle in the oriented graph $\bar G$.  Hence, $Q_\tau = 0$ unless $\tau(j_{\ell}) = i_{\ell}$ and $( \tau^k(i_{\ell}), k \geq 0 )$ is a path in $\bar G$.  We define $\sigma ( i_\ell) = \tau^{p_{\ell}} (i_{\ell})$ as the first element in $J$ which is met in the path. We may decompose these paths into disjoints path $\pi_\ell =( \tau^k (i_{\ell}), 0 \leq k \leq p_\ell )$ in $G$ from $i_{\ell}$ to $j_{\sigma ( \ell)}$.  It defines a path matching $\Pi =\{\pi_1, \cdots, \pi_b\}$. The contribution to $Q_{\tau}$ of any cycle of length at least $2$ is $1$ (since off-diagonal entries of $A$ and $B$ are $0$ or $1$). Also, the signature of disjoint cycles is the product of their signatures. So finally, it follows that
\begin{equation}\label{eq:detPIJ}
P_{I,J} (A) = \sum_{\Pi} \varepsilon(\Pi) \det ( B (x)_{\Pi, \Pi} ) = \sum_{\Pi} \varepsilon(\Pi) \det  ( ( A - x)_{\Pi, \Pi} ),
\end{equation}
where the sum is over all path matchings from $I$ to $J$ and $\varepsilon(\Pi)$ is the signature of the permutation $\tau$ on $\Pi$ defined by, if $\Pi = \{\pi_1, \cdots, \pi_b\}$, $\pi_\ell =( i_{\ell , 1}, \cdots, i_{\ell , p_\ell}) $ and $\sigma$ is the matching map of $\Pi$ :  for $1 \leq k \leq p_{\ell} -1$, $\tau ( i_{\ell,k} ) = i_{\ell, k+1} $ and $\tau ( i_{\ell, p_{\ell}} ) = \tau ( j_{\sigma(\ell)} ) = i_{\sigma(\ell)} $.

Observe that $ \det  ( ( A - x)_{\Pi, \Pi} )$ is a polynomial of degree $n - |\Pi|$ and leading coefficient $(-1)^{n - |\Pi|}$. Recall also that the signature of a cycle of length $k$ is $(-1)^{k+1}$. By assumption,  if $\Pi$ is a minimal path matching then its matching map is the identity : it follows that
$$\varepsilon(\Pi) =  (-1)^{n - \ell + b}.$$
Hence, from \eqref{eq:detPIJ},  $P_{I,J} (A)$ is a polynomial of degree $\ell$ and leading coefficient  $m (-1)^b$ where $m$ is the number of minimal path matchings.   By assumption $\Delta_b (A)$ divides $P_{I,J} (A)$ in particular $\deg (\Delta_b (A) ) \leq \ell$. It remains to apply Lemma \ref{le:detminor}. \ep

\paragraph{Vertical percolation (revisited).} Let us revisit the example of vertical percolation on $\dZ^2$ introduced in the previous paragraph.  We consider the graph $\Lambda_n$ on the vertex set $[0,n-1]^2 \cap \dZ^2$ where some vertical edges $\{ (x,y),(x,y+1)\}$ have been removed. We set $I = \{(0,0), (0,1), \cdots, (0,n-1)\}$ and $J =  \{(n-1,0), (n-1,1), \cdots, (n-1,n-1)\}$. Consider the path matchings from $I$ to $J$. Since none of the horizontal edges of the graph of $\dZ^2$ have been removed, the minimal path matching is unique, it matches $(0,k)$ to $(n-1,k)$ along the path $( (0,k), (1,k), \cdots , (n-1,k))$. In particular, the length of the minimal path matching is $n^2$. We may thus apply Theorem \ref{th:tool1} : we find that the multiplicity of any eigenvalue is bounded by $n = o(n^2)$.  By pointwise convergence of atoms, this implies that the  limiting spectral measure is continuous.   Note that Theorems \ref{th:tool2f} and \ref{th:tool1} give the same bound on the multiplicities for this example.

\paragraph{Lamplighter group. } The assumption that all minimal path matchings have the same matching map is important in the proof of Theorem \ref{th:tool1}. It is used to guarantee that the polynomial in \eqref{eq:detPIJ} is not identically zero. Consider a F\o lner sequence $B_n$ in the Cayley graph of the  lamplighter group $\dZ_2 \wr \dZ$ \cite{MR1866850} where $B_n$ consists of the vertices of the form $(v, k) \in \dZ_2 ^{\dZ} \times \dZ$ with $v(i) = 0$ for $|i| > n$ and $ |k| \le n$. There is an obvious minimal matching in $B_n$  covering all the vertices  where each path  is obtained  by shifting  the marker from $-n$ to $n$ keeping the configurations of the lamps unaltered along the way.  But
 the condition on the unicity of the matching map is not fulfilled. In this case, it is not hard to check that there is a perfect cancellation on the right hand side of  \eqref{eq:detPIJ}. It is consistent with the fact that spectral measure of this lamplighter group is purely atomic.

\section{Supercritical bond percolation on $\dZ^2$}
\label{sec:perco}

In this section, we will prove Theorem~\ref{thm:supercric_perc_Z2} by finding an explicit lower bound on the   total mass of
the continuous part of $\mu_\rho$ in terms of the speed of  the  point-to-point first passage percolation on $\dZ^2$.  We fix $p> p_c(\dZ^2)=1/2$.

We will use a finite approximation of $\dZ^2$.  Let $\Lambda_n(p)$ be the (random) subgraph of  the lattice $\dZ^2$  obtained by restricting the   $p$-percolation on  $\dZ^2$  onto the $(n+1) \times (n+1)$ box $[0, n]^2 \cap \dZ^2$. We simply write $\Lambda_n$ for $\Lambda_n(1)$. As mentioned in the introduction, $\mathrm{perc}(\dZ^2, p)$ is the  local weak limit of $U(\Lambda_n(p))$ and hence by Proposition~\ref{prop:defspecmeas},  we have that $\dE \mu^p_n$ converges weakly to $ \mu_\rho$ as $n \to \infty$,
where $\mu_n^p$  is the empirical eigenvalue  distribution  of $\Lambda_n(p)$ and the average $\dE$ is taken w.r.t.\ the randomness of $\Lambda_n(p)$.

Now, assume that, given a realization of the random graph $\Lambda_n(p)$, we can find  two disjoint subsets of vertices $U$ and $V$  of $\Lambda_n(p)$ with $|U| = |V|$ and a minimal vertex-disjoint path matching  $M_n$ of  $\Lambda_n(p)$  between $U$ and $V$ such that
\begin{enumerate}[(i)]
\item The vertices of $U$ and $V$ are uniquely paired up in any such minimal matching of  $\Lambda_n(p)$ between $U$ and $V$.
\item $|U| = o(n^2)$.
\item There exists a  constant $c>0$ such that  the size of $M_n$  is at least $cn^2$,  with probability converging to one.
\end{enumerate}
If such a matching exists satisfying property (i), (ii) and (iii) as above,  then Theorem~\ref{th:tool1} says that for any finite subset $S \subset \dR$,
\[ \dP(\mu_n^p(S) \le 1-c)  = 1 - o(1),\]
and consequently, $ \dE \mu^p_n (S) \le (1-c) + o(1)$. Then by L\"{u}ck approximation (see \cite[Corollary 2.5]{MR2148799}, \cite[Theorem 3.5]{thom08} or \cite{ATV})  $\mu_\rho(S) = \lim_{n \to \infty}   \dE \mu^p_n (S) \le 1-c$ for any finite subset $S$,  which implies that the   total mass of
the continuous part of $\mu_\rho$ is at least $c$. Hence, in order to prove Theorem~\ref{thm:supercric_perc_Z2}, it is sufficient to prove the existence with high probability of such pair of disjoint vertices.

A natural way to construct   this is to find a linear number of vertex-disjoint paths   in $\Lambda_n(p)$  between its left and right boundary.
 Suppose that  there exists a  collection of $m$ disjoint   left-to-right crossings of $\Lambda_n(p)$ that matches the vertex $(0, u_i)$ on the left boundary to the vertex $(n, v_i)$ on the right boundary for $1 \le i \le m$. Without loss of generality, we can assume $0 \le u_1 < u_2 < \cdots < u_{m} \le n$.  Since two vertex-disjoint left-to-right crossings in $\dZ^2$  can never cross each other, we always have $0 \le v_1 < v_2 < \cdots < v_{m} \le n$. Now we  take $U  = \{ (0, u_i) : 1 \le i \le m\}$ and $V = \{ (n,  v_i) : 1 \le i \le m \}$.  We consider  all vertex-disjoint path matchings between $U$ and $V$ in $\Lambda_n(p)$ (there exists at least one such matching  by our hypothesis) and take $M_n$ to be a minimal matching between $U$ and $V$.  Clearly, the property  (i) and (ii) above are satisfied. Since any   left-to-right crossing contains  at least $(n+1)$ vertices, the size of $M_n$ is at least $(n +1) m$. Thus to satisfy the property (iii) we need to show that with high probability we can find at least $cn$ many vertex-disjoint left-to-right crossings in $\Lambda_n(p)$.

 Towards this end, let $\ell_n$ denote the maximum number of vertex-disjoint paths   in $\Lambda_n(p)$  between its left and right boundary. By Menger's theorem, $\ell_n$ is also equal to the size of a minimum vertex cut of $\Lambda_n(p)$, that is, a set of vertices of smallest size that  must be removed to disconnect the left and right boundary of $\Lambda_n(p)$.  Note that to bound  $\ell_n$ from below, it suffices to find a lower bound on  the size of a minimum edge cut  of $\Lambda_n(p)$, since the size of a minimum edge cut  is always  bounded above by  $4$ times the size of a minimum vertex cut. This is because deleting all the edges incident to the vertices in a minimum vertex cover gives an edge cut. The reason behind considering minimum edge cut instead of  minimum vertex cut is that the size of the former can be related to certain line-to-line first passage time in the dual graph of $\Lambda_n$, whose edges are weighted by i.i.d.\ $\mathrm{Ber}(p)$. We describe this connection below.

  Let $\Lambda_n^*$ (called the dual of $\Lambda_n$)  be a graph with vertices $\{ (x+ \frac12, y + \frac12):  0 \le x \le n-1, - 1 \le y \le n\} $, with all edges of connecting the pair of vertices with $\ell_1$-distance exactly $1$, except for those in top and bottom sides.  To each edge $e$ of $\Lambda_n^*$, we assign a random weight  of value $1$ or $0$ depending on whether the unique edge of $\Lambda_n$,  which $e$ crosses,  is present or absent  in the graph $\Lambda_n(p)$. Hence, the edge weights of $\Lambda_n^*$ are i.i.d.\ $\mathrm{Ber}(p)$. Now here is the crucial observation. The size of minimum edge cut of $\Lambda_n(p)$, by duality, is same as  the minimum weight  of a path from the top to bottom boundary of $\Lambda_n^*$. Moreover, since the dual lattice of $\dZ^2$ is isomorphic to $\dZ^2$, the minimum weight of a top-to-bottom crossing  in $\Lambda_n^*$ is equal in distribution to the line-to-line passage time  $t_{n+1,n-1}(\mathrm{Ber}(p))$ in $\dZ^2$, where
  \begin{align*}
   t_{n, m}(F) := \inf \Big\{ \sum_{e \in \gamma} t(e):  &\gamma \text{ is a path in $\dZ^2$  joining } (0, a), (n, b) \text{ for some }  0 \le a, b \le m \\
   &\text{ and  $\gamma$ is contained in } [0, n] \times [0, m] \Big\},
   \end{align*}
   and $t(e)$, the weight of edge $e$ of $\dZ^2$, are i.i.d.\  with nonnegative distribution $F$. By Theorem 2.1(a) of \cite{GrimmettKesten}, for any nonnegative distribution $F$,  we have
   \begin{equation}\label{FPP_lower_bound}
    \liminf_{n \to \infty} \frac1n t_{n, n}(F) \ge \nu(F) \ \  a.s.,
    \end{equation}
  where $\nu(F) < \infty $ is called the speed (or time-constant) of the first passage percolation on $\dZ^2$ with i.i.d.\  $F$ edge weights, that is,
  \[ \frac 1n a_{0, n}(F)  \to \nu(F)\ \ \text{ in probability}, \]
  where
 \begin{align*}
   a_{0, n}(F) := \inf \Big\{ \sum_{e \in \gamma} t(e):  &\gamma \text{ is a path in $\dZ^2$  joining } (0, 0), (n, 0)  \Big\}.
   \end{align*}
It is a classical fact due to Kesten \cite{Kesten} that the speed is strictly positive  or $\nu(F)>0$ if and only if $F(0)< p_c(\dZ^2) = \frac12$.  This ensures that $\nu(\mathrm{Ber}(p)) >0$ in the supercritical regime $p > \frac12$.
Therefore, for any $\veps >0$, with probability tending to one,
  \[ t_{n+1,n-1}(\mathrm{Ber}(p)) \ge t_{n+1,n+1}(\mathrm{Ber}(p)) \ge \big(\nu(\mathrm{Ber}(p))  - \veps\big) (n+1),\]
  which implies that
  \[ \lim_{n \to \infty} \mathbb {P} \left(\ell_n \ge \frac14 \big(\nu(\mathrm{Ber}(p))  - \veps\big) n \right)=1. \]
Hence  the property (3) is satisfied with $c = \frac14 \big(\nu(\mathrm{Ber}(p))  - \veps\big) $ for any $\veps>0$. Therefore, the total mass of the continuous part of $\mu_\rho$ is bounded below by $  \frac14 \nu(\mathrm{Ber}(p)) $.

This concludes the proof of Theorem~\ref{thm:supercric_perc_Z2}. \ep

\section{Spectrum of unimodular trees}
\label{sec:trees}

\subsection{Stability of unimodularity}
\label{subsec:stab}

In the sequel, we will use a few times that unimodularity is stable by weights mappings, global conditioning and invariant percolation. More precisely, let $ (G, o)$ be a unimodular random weighted rooted graph with distribution $\rho$. The weights on $G$ are denoted by $\omega : V^2 \to \dZ$. The following trivially holds :

{\em Weight mapping : }
let $\psi : \cG^* \to \dZ$ and $\phi : \cG^{**} \to \dZ$ be two measurable functions. We define $\bar G$ as the weighted graph with weights $\bar \omega$, obtained from $G$ by setting for $u \in V$, $\omega (u,u) = \psi (G,u)$ and for $u,v \in V^2$ with $\{u,v\} \in E(G)$, $\omega(u,v)  = \phi ( G,u,v)$. The  random rooted weighted graph $ (\bar  G,o)$ is unimodular. Indeed, the $\cG^* \to \cG^*$ map $G \mapsto \bar G$ is measurable and we can apply \eqref{eq:defunimod}  to $f (G,u,v) = h ( \bar G,u,v)$ for any measurable $h : \cG^{**} \to \dR_+$.

{\em Global conditioning : }
let $A$ be a measurable event on $\cG^*$ which is invariant by re-rooting: i.e. for any $ (G,o)$ and $( G',o)$ in $\cG^*$ such that $G$ and $ G'$ are isomorphic, we have $(G,o) \in A$ iff $( G',o) \in A$. Then,  if $\rho ( A) >0$,  the random rooted weighted graph $ (G,o)$ conditioned on $ ( G,o) \in A$ is also unimodular (apply \eqref{eq:defunimod}  to $f (G,u,v) = \IND ( (G,u) \in A ) h ( G,u,v)$ for any measurable $h : \cG^{**} \to \dR_+$).

{\em Invariant percolation : } let $B\subset \dZ$. We may define a random weighted graph $\hat G$ with edge set $E(\hat G) \subset E (G)$ by putting the edge $\{u,v\} \in E( G)$ in $E(\hat G)$ if both $\omega(u,v)$ and $\omega(v,u)$ are  in $B$. We leave the remaining weights unchanged. Then the random weighted rooted graph $(\hat G(o),o)$ is also unimodular (apply \eqref{eq:defunimod}  to $f (G,u,v) = h ( \hat G(u) ,u,v)$ for any measurable $h : \cG^{**} \to \dR_+$). 

As an application the measure $\rho'$ defined  in the statement of Theorem \ref{th:treesILE} is unimodular. Indeed, consider the weight mapping for $v \in V$, $\omega(v,v) = \IND ( v \in L)$ and for $\{u,v\} \in E$, $\omega(u,v) = \omega (v,u) = \IND ( \omega(u,u) =  \omega(v,v))$. Then we perform an invariant percolation with $B = \{1\}$ and finally a global conditioning by $A = \{\hbox{all vertices in $G$ satisfying $\omega(v,v) = 0$}\}$.

\subsection{Proof of Theorem \ref{th:treesILE}}

Consider the unimodular weighted tree $(T,L,o)$.  Our main strategy will be to construct a suitable  invariant labeling on $T$ using the invariant line ensemble $L$ and then apply Theorem~\ref{th:tool2}.

We may identify $L$ as a disjoint union of countable lines $(\ell_i)_{i}$. Each such line $\ell \subset L$ has two topological ends. We enlarge our probability space and associate to each line an independent Bernoulli variable with parameter $1/2$. This allows to orient each line $\ell \subset L$.  This can be done by choosing the unique  vertex on the line $\ell$ whose distance from the root $o$ is minimum and then by picking one of its two neighbors on $\ell$ using the Bernoulli coin toss.

Let us denote by $(\overrightarrow {\ell_i})_i$ the oriented lines. We obtain this way a unimodular weighted graph $(T,\omega,o)$ where $\omega(u,v) = 1$ if the oriented edge $(u,v) \in \overrightarrow {\ell_i}$ for some $k$, $\omega(u,v) = -1$ if $(v,u) \in  \overrightarrow {\ell_i}$, and otherwise $\omega(u,v) = 0$.

Now, we fix some integer $k \geq 1$. There are exactly $k$ functions $\eta : V \mapsto \dZ/{k \dZ}$ such that  the discrete gradient of $\eta$ is equal to $\omega$ (i.e.\ such that for any $u,v \in V$ with $\{u,v\} \in E$, $\eta (u) - \eta (v) = \omega (v,u)$ $\mathrm{mod}(k)$) since given the gradient $\omega$, the function $\eta$ is completely determined by its value at the root.  We may enlarge our probability space in order to sample, given $(T,\omega,o)$, such a function $\eta$ uniformly at random. Then the vertex-weighted random rooted graph $(T,\eta,o)$ is unimodular.

In summary, we have obtained an invariant labelling $\eta$ of $(T,o)$ such that all vertices $v \in V$ outside $L$ are level, all vertices in $L$ such that $\eta (v) \ne 0$ are prodigy, and vertices in $L$ such that $\eta (v) = 0$ are bad. By Theorem \ref{th:tool2}, we deduce that for any real $\lambda$,
$$
\mu_\rho (\lambda) \leq \dP( o \hbox{ is bad}) + \sum_{j} \ell_j,
$$
where $\ell_j = \dE \langle e_o , P_j e_o \rangle $ and $P_j$ is the projection operator of the eigenspace of $\lambda$ in the adjacency operator $A_j$ spanned by vertices with label $j$. Now, observe that the set of level vertices with label $j$ are at graph distance at least $2$ from the  set of level vertices with label $i \ne j $. It implies that the operators $A_j$ commute and $A'$, the adjacency operator of $T' = T \backslash L$, can be decomposed as a direct sum of the operators $A_j$. It follows that, if $P'$ is the projection operator of the eigenspace of $\lambda$ in $A'$
$$
\sum_j \ell_j =  \dE \langle e_o , P' e_o \rangle  =  \dP( o \notin L) \mu_{\rho'} (\lambda).
$$
Also, by construction, $\dP( o \hbox{ is bad}) $ is upper bounded by $1/k$. Since $k$ is arbitrary, we find
$$
\mu_\rho (\lambda) \leq \dP( o \notin L) \mu_{\rho'} (\lambda).
$$
This concludes the proof of  Theorem \ref{th:treesILE}. \ep

\begin{remark}\textnormal{
In the proof of Theorem \ref{th:treesILE}, we have used our tool Theorem \ref{th:tool2}. It is natural to ask if we could have used Theorem \ref{th:tool1} together with some finite graphs sequence $(G_n)$ having local weak limit $(T,o)$ instead. We could match the set of $v \in L$ such that $\eta(v) = 1$ to the set of $v \in L$ such that $\eta(v) = k-1$ forbidding the set of $v \in L$ with $\eta (v)= 0$. Note however that if the weighted graph $(G_n, \eta_n)$ has local weak limit $(T,\eta,o)$ then the boundary of $\eta_n^{-1} (j)$ for $j \in  \dZ/{k \dZ}$ has cardinal $(2/k + o(1))\dP(o \in L) |V(G_n)|$. In particular, the sequence $(G_n)$ must have a small Cheeger constant. It implies for example that when $p(0)=p(1)=0$  we could not use the usual random graphs as finite approximations of infinite unimodular Galton-Watson trees since they have a Cheeger constant bounded away from $0$, see Durrett \cite{MR2656427}. }
\end{remark}

\subsection{Construction of invariant line ensemble on unimodular tree}

We will say that a unimodular tree $(T,o)$ is {\em Hamiltonian} if there exists an invariant line ensemble $L$ that contains the root $o$ with probability 1. As the first example, we show that $d$-regular infinite tree is Hamiltonian.

\begin{lemma}\label{le:regHam}
For any integer $d \geq 2$, the $d$-regular infinite tree is Hamiltonian.
\end{lemma}

\bp{} The case $d=2$ is trivial : in this case $T = (V,E)$ itself is a line ensemble.  Let us assume $d \geq 3$. On a probability space, we attach to each oriented edge $(u,v)$ independent variables, $\xi(u,v)$ uniformly distributed on $[0,1]$. With probability one, for each $u \in V$, we may then order its $d$ neighbours according to value of $\xi(u,\cdot)$.  This gives a weighted graph $(T,\omega,o)$ such that, for each $u \in V$ with $\{u,v\} \in V$,  $\omega (u,v) \in \{1,\cdots, d\}$ is the rank of vertex $v$ for $u$. Note that $\omega(u,v)$ may be different from $\omega(v,u)$.  We now build a line ensemble as follows. The root picks its first two neighbours, say $u_1$, $u_2$, and we set $L(u_1, o) =  L(u_2, o) = 1$, for its other neighbours, we set $L(u,o) = 0$. To define further $L$, let us introduce some notation. For $u \ne v$, let $T^v_u$ be the tree rooted at $u$ spanned by the vertices whose shortest path in $T$ to $v$ meets $u$, and let $a^v(u)$ be the first visited vertex on the shortest path from $u$ to $v$ (see Figure \ref{fig:ancestor}). Then,  we define iteratively the line ensemble  (we define $L(u, \cdot)$ for a vertex $u$ for which $L(a^o(u), \cdot)$ have already been defined)  according to the rule :
 if $L(u,a^o(u)) =1$ then $u$ picks its first neighbour in $T^o_u$, say $v_1$, and we set $L(u,v_1) = 1$, otherwise $L(u,a^o(u)) = 0$ and $u$ picks its two first neighbours in $T^o_u$, say $v_1,v_2$, and we set $L(u,v_1) = L(u,v_2) = 1$. In both cases, for the other neighbours of $u$ in $T^o_u$, we set $L(u,v) =0$.

\begin{figure}[htb]
\centering\begin{psfrags}
 \psfrag{u}{$u$}
  \psfrag{v}{$v$}
  \psfrag{Tu}{$T^v_u$}
  \psfrag{Tv}{$T^u_v$}
  \psfrag{a}{$a^v(u)$}
  \includegraphics[width=12cm]{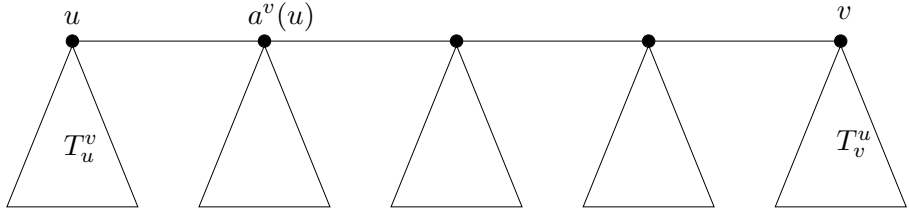}
\end{psfrags}
\caption{Definition of $a^v(u)$ and $T^v _u$.}\label{fig:ancestor}
\end{figure}

Iterating this procedure gives a line ensemble which covers all vertices. It is however not so clear that this line ensemble is indeed invariant since, in the construction, the root seems to play a special role. In order to verify \eqref{eq:defunimod}, it is sufficient to restrict to functions $f(G,L,u,v)$ such that $f(G,L,u,v) = 0$ unless $\{u,v\} \in E$ (see \cite[Proposition 2.2]{aldouslyons}). Let us denote $v_1, \cdots, v_d$ the neighbours of the root, we
have
$$
\dE \sum_{k=1} ^ d f ( T,L,o,v_k) = (d-2) \dE [ f (T,L,o,v_1) | L(v_1,o) = 0 ] + 2 \dE [ f (T,L,o,v_1) | L(v_1,o) = 1].
$$
We notice that the rooted trees $T^v_u$, $u \ne v$, are isomorphic ($T^v_u$ is a $(d-1)$-ary tree) and that, given the value of $L(u,v_1)$, the restriction of $L$ to $T^o_{v_1}$ and $T^{v_1}_o$ have the same law (and are independent). Since $L(u,v) = L(v,u)$, it follows that, for $\varepsilon \in \{0,1\}$,
$$
\dE [ f (T,L,o,v_1) | L(v_1,o) = \varepsilon ] = \dE [ f (T,L,v_1,o) | L(o,v_1) = \varepsilon ].
$$
We have thus checked that $L$ is an invariant line ensemble. \ep

\begin{lemma}\label{l:2ktree}
Let $k\ge 3$. Every unimodular tree with all degrees either $2$ or $k$ has an invariant line ensemble of density $\dE \,\deg(o)/k$.
\end{lemma}
\bp{}
Sample the unimodular random tree $(T,o)$.
Consider the $k$-regular labeled tree $T'$ that one gets by contracting each induced subgraph which is a path to a single edge labeled by the number of vertices.  This tree has an invariant line ensemble $L'$ with density 1; this corresponds to a line ensemble $L$ in $T$. Since each edge in $T'$ is contained in $L'$ with probability $2/k$, it follows that each edge of $T$ is contained in $L$ with probability $2/k$. Thus the expected degree of $L$ at the root of $T$ given $T$ is $\frac{2}{k}\deg(o)$. The claim follows after averaging over $T$.
\ep
\bigskip

The following proves Proposition \ref{prop:ILE}, part 2 for the case $q=0$ (i.e.\ when there are no ``bushes"). 

\begin{proposition}\label{p:le}
Let $T$ be a unimodular tree with degrees in $\{2,3,\ldots, d\}$. Then $T$ contains an invariant line ensemble with density
at least $\frac{1}{3}\dE\,\deg(o)/d$. In fact, when $d\ge 6$ the density is at least  $\frac{1}{3}\dE\,\deg(o)/(d-4)$.
\end{proposition}

A tree constructed of $d$-stars with paths of length $m$ emanating shows that in some cases the optimal density can be arbitrary close to $\dE\, \deg(o)/d$. In this sense our bound is sharp up to a factor of $1/3$.

\bigskip

\bp{ of Proposition \ref{p:le}}
If $d \geq 6$ we argue as follows. For each $k$, we split all vertices of degree $3k+2j$ with $j=0,1,2$ into $k$ groups of vertices of degree $3$ and $j$ groups of vertices of degree $2$. We can perform this in an unimodular fashion by ordering the adjacent edges of a vertex uniformly at random (see the proof of Lemma \ref{le:regHam}).  This way we obtain a countable collections of trees $(T_n)_{n\geq 1}$.

By Lemma \ref{l:2ktree} each of these trees contains invariant line ensembles with expected degree $\frac{2}{3}\ev \deg_{T_n}(o)$. In particular, the expected degree of their union $F_1$ in $T$ is $\frac{2}{3}\dE \deg(o)$. We thus have found an invariant  subforest $F_1$ of $F_0=T$ with degrees in $\{0,2,4\ldots, 2k+2j\}$ and expected degree $\frac{2}{3}\dE \deg(o)$.

Iterating this construction $i$ times we get a sequence of subforests $F_i$ with expected degree $\left(\frac{2}{3}\right)^i \dE\, \deg(o)$. The maximal degree of $F_i$ is bounded above by some $d_i$ (with $d_0=d$), which satisfy the following recursion: if $d_i=3k+2j$ with $j=0,1,2$, then $d_{i+1}=2k+2j$. In particular, $d_i$ is even for $i\ge 1$, and
\begin{equation}\label{e:rec}
d_{i+1}\le \frac{2}{3}d_i+\frac{4}{3}.
\end{equation}

Let   $k$ be the first value so that $d_k\le 4$; by checking cases we see that $d_k=4$, and that
$d_{k-1}=5$ or $d_{k-1}=6$.  Assuming $k>1$ we also know that $d_{k-1}$ is even, so $d_{k-1}=6$. Otherwise, $k=1$ and then $d_0 = d$. However the assumption $d \ge 6$ yields to $d_0 = d = 6$. Hence in any case $d_{k-1} = 6$. Now using the inequality \eqref{e:rec} inductively we see that for $1 \le i\le k$ we have $d_{k-i}\ge \frac{4}{3}\left(\frac{3}{2}\right)^i+4$.
Setting $i=k$ and rearranging we get
$$
\left(\frac{2}{3}\right)^k\ge \frac{4}{3}\frac1{d-4}.
$$
The forest $F_k$ has degrees in $\{0,2,4\}$. Another application of Lemma \ref{l:2ktree} (with $k=4$ there)
gives an invariant line ensemble with density
\[
\frac{1}{4} \left(\frac{2}{3}\right)^k \dE\, \deg(o) \ge \frac{1}{3}\frac{\dE\,\deg(o)}{d-4}.
\]

If $d = 5$, then $k =1$, and the above argument gives an invariant line ensemble with density $ \frac{1}{4} \left(\frac{2}{3}\right) \dE\, \deg(o)$.

The only cases left are $d=3,4$. In the first case, just use Lemma \ref{l:2ktree} with $k=3$. In the second, split each degree 4 vertex in 2 groups of
degree 2 vertices as above. Then apply Lemma \ref{l:2ktree} with $k=3$ to get a subforest with degrees in ${0,2,4}$. Then apply the Lemma again with $k=4$. The density lower bounds are given by
$
\frac{1}{3}\dE \,\deg(o)$, $\frac{1}{6}\dE \,\deg(o)
$
respectively, and this proves the remaining cases. \ep
\bigskip

Recall that the core $C$ of a  tree $T$ is the induced subgraph of vertices such that  removal of each vertex in $C$  breaks $T$ into at least two infinite components.
The following is a reformulation of part (ii) of Proposition \ref{prop:ILE}.

\begin{corollary}[Removing bushes]\label{c:rbushes}
Let $(T,o)$ be an infinite  unimodular tree, with core $C$ and maximal degree $d$. Then Proposition \ref{p:le} holds with $\ev \deg(o)$ replaced by $\ev \deg(o)- 2\dP(o\notin C)$.
\end{corollary}

\begin{proof}
We clarify that $\deg_C(o)=0$ if $o\notin C$.  It suffices to to show that $\ev \deg_C(o) =\dE \deg(o)-2\dP(o\notin C)$. For this, let every vertex $v$ with $\deg_C(v) = 0$ send unit mass to the unique neighbor vertex closest to $C$ (or closest to the single end of $T$ in case $C$ is empty).
We have
$$
\deg_C(o) = \deg(o) - r -\one(o\notin C)
$$
where $r$ is the amount of mass $o$ receives. The claim now follows by mass transport : \eqref{eq:defunimod} applied to $f (G,o,v)$ equal to the amount of mass send by $o$ to $v$ gives $ \dP ( o \notin C) = \dE r$.
\end{proof}

We are now ready to prove the main assertion of Proposition \ref{prop:ILE}, repeated here as follows.

\begin{corollary}
Let $(T,o)$ be a unimodular tree with at least $2$ ends with positive probability. Then $T$ contains an invariant line ensemble with
positive density.
\end{corollary}
\bp{}
We may decompose the measure according to whether $T$ is finite or infinite and prove the claim separately.
The finite case being trivial, we now assume that $T$ is infinite.

Consider the core $C$ of $T$. If $T$ has more than one end, then $C$ has the same ends as $T$, in particular it is not empty. Thus for the purposes of this corollary we may assume that $T=C$, or in other words all degrees of $T$ are at least 2.

If  $\ev \,\deg(o)=2$, then $T$ is a line and we are done.  So next we consider the case  $\ev \deg(o)>2$.

Let $F_d$ be a subforest where all edges incident to vertices of degree more
than $d$ are removed. Then
$\deg_{F_d}(o)\to \deg_T(o)$ a.s.\ in a monotone way. Thus by the Monotone Convergence Theorem $\ev \deg_{F_d}(o) \to \ev \deg_T(o)>2$. Pick a $d$ so that
$\ev \deg_{F_d}(o) >2$.
Corollary \ref{c:rbushes} applied to the components of $F_d$ now yields the claim.
\ep
\bigskip

Part (i) of Proposition \ref{prop:ILE} is restated here as follows.

\begin{corollary}
Let $T$ be a unimodular tree and assume that $\dE \deg(o)^2$ is finite.
Then $T$ contains an invariant line ensemble $L$ with density
$$
\dP ( o \in L) \ge \frac{1}{6}\frac{(\ev\, \deg(o)-2)_+^2}{\ev\, \deg(o)^2}.
$$
\end{corollary}

\bp{} Let $d \geq 1$ be an integer. For each vertex $v$ we mark $(\deg(v)-d)_+$ incident edges at random. To set up a mass transport argument,
we also make each vertex to send mass one along every one of its marked edges. The unmarked edges form a forest $F_d$ with the same vertices as $T$ and maximal
degree $d$: we now bound its expected degree. Note that the degree of the root in $F_d$ is bounded below by
the  same in $T$ minus the total amount of mass sent or received. These two quantities are equal in expectation, so we get
$$
\ev\,\deg_{F_d}(o)\ge \ev\,\deg(o) -2\ev(\deg(o)-d)_+.
$$
By Proposition \ref{p:le} applied to components of $F_d$, as long as $d\ge 6$ we get an invariant line ensemble $L$ with density
$$
\dP(o\in L)\ge \frac{1}{3}\frac{1}{d-4} \left(\ev\,\deg(o) -2-2\ev(\deg(o)-d)_+\right).
$$
To bound the last term, note that setting $c=\deg(o)-d$, the inequality
$
4(\deg(o)-d)_+ d \le \deg(o)^2
$
reduces to $4cd \le (c+d)^2$, which certainly holds. Thus we can bound
$$
\dP(o\in L) \ge \frac{1}{3} \frac{1}{d-4}\left(\ev\,\deg(o) -2-\frac{\ev\deg(o)^2}{2d}\right).
$$
Now set $d=\lceil \ev \deg(o)^2/(\eta-2)\rceil \ge \eta^2/(\eta-2)\ge 8$, where $\eta=\dE \deg(o)$ can be
assumed to be more than  $2$.
Using the bound $\lceil x\rceil - 4 \le x$ we get the claim.
\ep

\subsection{Maximal invariant line ensemble}

Let $(T,o)$ be a unimodular rooted tree with distribution $\rho$. In view of Theorem \ref{th:treesILE} and Proposition \ref{prop:ILE}, we may wonder what it is the value
$$
\Sigma(\rho) = \sup \dP ( o \in L ),
$$
where the supremum runs over all invariant line ensembles $L$ of $(T,o)$. Recall that a line ensemble $L$ of $(T,o)$ is a weighted graph $(T,L,o)$ with weights $L(u,v)$ in $\{0,1\}$. By diagonal extraction, the set of $\{0,1\}$-weighted graphs of a given (locally finite) rooted graph $G= (G,o)$ is compact for the local topology. Hence, the set of probability measures on rooted $\{0,1\}$-weighted graphs such that the law of the corresponding unweighted rooted graph is fixed is a compact set for the local weak topology. Recall also that the set of unimodular measures in closed for the local weak topology. By compactness, it follows that there exists an invariant line ensemble, say $L^*$, such that
$$
\Sigma(\rho) = \dP ( o \in L^*).
$$
It is natural to call such invariant line ensemble a maximal invariant line ensemble.

\begin{op}\label{q:mi1}
What is the value of $\Sigma(\rho)$ for $\rho$ a unimodular Galton-Watson tree ?
\end{op}

Let $L^ *$ be an maximal invariant line ensemble and assume $\dP( o \in L^ *) < 1$. Then $\rho'$, the law of $(T\backslash L^ * ,o)$ conditioned on $o \notin L^ *$, is unimodular. Assume for simplicity that $\rho$ is supported on rooted trees with uniformly bounded degrees. Then, by Proposition \ref{prop:ILE} and the maximality of $L^*$, it follows that, if $(T',o)$ has law $\rho'$, then a.s.\ $T'$ has either $0$ or $1$ topological end. Theorem \ref{th:treesILE} asserts that the atoms of $\mu_{\rho}$ are atoms of $\mu_{\rho'}$. We believe that the following is true.

\begin{op}\label{q:mi2}
Is it true that if $\rho$ is a unimodular Galton-Watson tree then $\rho'$ is supported on finite rooted trees ?
\end{op}

\subsection{Two examples}

\paragraph{Ring graphs.}

With Theorem \ref{th:treesILE}, we can give many examples of unimodular rooted trees $(T,o)$ with continuous expected spectral measure. Indeed, by Theorem \ref{th:treesILE} all Hamiltonian trees have continuous spectrum.

An example of a Hamiltonian unimodular tree is the {\em unimodular ring tree} obtained as follows. Let $P \in \cP(\dZ_+)$ with finite positive mean. We build a multi-type Galton-Watson tree with three types $\{o,a,b\}$. The root $o$ has type-$o$ and has two type-$a$ children and a number of type-$b$ children sampled according $P$. Then, a type-$b$ vertex has a $2$ type-$a$
children and a number of type-$b$ sampled independently according to $\wP$ given by \eqref{eq:defwP}. A type-$a$ vertex has $1$ type-$a$ child and a number of type-$b$ sampled according to $P$. We then remove the types and obtain a rooted tree. By construction, it is Hamiltonian : the edges connecting type-$a$ vertices to their genitor is a line ensemble covering all vertices. We can also check easily that it is unimodular.

If $P$ has two finite moments, consider a graphic sequence $\underline d(n) = (d_1(n), \cdots , d_n(n))$ such that the empirical distribution of $\underline d(n)$ converges weakly to $P$ and whose second moment is uniformly integrable. Sample a graph $G_n$ with vertex set $\dZ / (n \dZ)$ uniformly on graphs with degree sequence $\underline d(n)$ and, if they are not already present, add the edges $\{k,k+1\}$, $k \in \dZ / (n \dZ)$. The a.s.\ weak limit of $G_n$ is the above ring tree.  This follows from the known result that the uniform graph with degree sequence $\underline d(n)$ has a.s. weak limit the unimodular Galton-Watson tree with degree distribution $P$ (see \cite{MR2656427,MR2643563,notesRG})

Alternatively, consider a random graph $G_n$ on $\dZ / (n \dZ)$ with the edges $\{k,k+1\}$, $k \in \dZ / (n \dZ)$ and each other edge is present independently with probability $c/n$. Then the a.s. weak limit of $G_n$ will be the unimodular ring tree with $P = \POI(c)$. Note that $G_n$ is the Watts-Strogatz graph \cite{wattsstrogatz}.

\paragraph{Stretched regular trees.}

Let us give another example of application of Theorem \ref{th:treesILE}. Fix an integer $d \ge 3$. Consider a unimodular rooted tree $(T,o)$ with only vertices of degree $2$ and degree $d$. Denote its law by $\rho$. For example a unimodular Galton-Watson tree with degree distribution $P = p \delta_2 + (1-p)\delta_d$, $0<p<1$. Then, arguing as in Proposition \ref{prop:ILE}, a.s., all segments of degree $2$ vertices are finite.  Contracting these finite segments, we obtain a $d$-regular infinite tree. Hence, by Lemma \ref{le:regHam}, there exists an invariant line ensemble $L$ of $(T,o)$  such that a.s.\ all degree $d$ vertices are covered. By Theorem \ref{th:treesILE}, the atoms of $\mu_\rho$ are contained in set of atoms in the expected spectral measure of rooted finite segments. Eigenvalues of finite segments of length $n$ are of the form $\lambda_{k,n} = 2 \cos (  \pi k / (n+1))$, $1 \leq k \leq n$. This proves that the atomic part of $\mu_\rho$ is contained in $\Lambda = \cup_{k,n} \{ \lambda_{k,n}\}  \subset (-2,2)$.

On the other hand, if $\rho$ is a unimodular Galton-Watson tree with degree distribution $P = p \delta_2 + (1-p)\delta_d$, $0<p<1$, the support of $\mu_{\rho}$ is equal to $[-2 \sqrt{d-1}, 2\sqrt{d-1}]$. Indeed, recall that $\mu_{\rho} = \dE_{\rho} \mu_A ^ {e_o}$ and
$$
\int x^{2k} \mu_A ^ {e_o} = \langle e_o , A^{2k} e_o \rangle
$$
is equal to the number of path in $T$ of length $2k$ starting and ending at the root. An upper bound is certainly the number of such paths in the infinite $d$-regular tree. In particular, from Kesten \cite{MR0109367},
$$
\int x^{2k} \mu_A ^ {e_o} \leq ( 2\sqrt{d-1} + o (1) )^ {2k} .
$$
It implies that the convex hull of the support of $\mu_{\rho}$ is contained $[-2 \sqrt{d-1}, 2\sqrt{d-1}]$. The other way around, recall first that if $\mu$ is the spectral measure of the infinite $d$-regular tree then $\mu(I) > 0$ if $I$ is an open interval in  $[-2 \sqrt{d-1}, 2\sqrt{d-1}]$, see \cite{MR0109367}. Recall also that for the local topology on rooted graphs with degrees bounded by $d$, the map $G \mapsto \mu^{e_o}_{A(G)}$ is continuous in $\cP(\dR)$ equipped with the weak topology (e.g. it follows from Reed and Simon  \cite[Theorem VIII.25(a)]{reedsimon}). Hence, there exists $t >0$ such that if $(T,o)_t$ is $d$-regular then $\mu_{A(T)}^{e_o} (I) >0$. Observe finally that under $\rho$ the probability that $(T,o)_t$ is $d$-regular is positive. Since $\mu_{\rho} = \dE_{\rho} \mu_A ^ {e_o}$, it implies that $\mu_{\rho} (I) >0$.

We thus have proved that for a unimodular Galton-Watson tree with degree distribution $P = p \delta_2 + (1-p)\delta_d$, $\mu_{\rho}$ restricted to the interval $[2,2\sqrt{d-1}]$ is continuous.

\section{Proof of Proposition \ref{prop:defspecmeas}}

Restricted to sofic measures, the proof of this proposition is contained in \cite{MR2789584}, \cite{MR2724665}. To bypass this limitation, we introduce some concepts of operator algebras.

Consider a Von Neumann algebra $\cM$ of bounded linear operators on a Hilbert space $H$ with a normalised trace $\tau$. If $A \in \cM$ is self-adjoint, and hence bounded, we denote by $\mu_A$ its spectral measure, i.e. the probability measure such that
$$
\tau (A ^k ) = \int x^k d\mu_A(x).
$$
The rank of $A$ is defined as
$$
\RANK ( A) = 1 - \mu_A(\{0\}).
$$
Recall that the Kolmogorov-Smirnov distance between two probability measures on $\dR$ is the $L^\infty$ norm of their partition functions :
$$
d_{KS} ( \mu, \nu ) = \sup_{t \in \dR} | \mu ( -\infty, t ] - \nu ( -\infty, t ] |.
$$
We have that $d_{KS} ( \mu,\nu) \geq d_L (\mu,\nu)$ where $d_L$ is the L\'evy distance,  
$$
d_{L} (\mu,\nu) = \inf \{\veps > 0 : \forall t \in \dR , \mu ( -\infty, t - \veps ] - \veps \leq  \nu ( -\infty, t] \leq \mu ( -\infty, t + \veps ] + \veps  \}, 
$$
(recall that the L\'evy distance is a metric for the weak convergence). We start with a simple lemma which is the operator algebra analog of a well known rank inequality (see e.g. Bai and Silverstein \cite[Theorem A.43]{MR2567175}).
\begin{lemma}\label{le:rankineq}
If $A,B \in \cM$ are self-adjoint,
$$
d_{KS} (\mu_A , \mu_B ) \leq \RANK ( A -B).
$$
\end{lemma}
\bp{}
We should prove that for any $J = (-\infty,t]$ we have
$| \mu_A ( J ) - \mu_B  (J) | \leq \RANK ( A -B)$. There is a convenient variational expression for $\mu_A(J)$ :
\begin{equation}\label{eq:minmaxJ}
\mu_A ( J ) = \max \{ \tau (P) :  P A P \leq t P , P \in \cP \},
\end{equation}
where $\cP \subset \cM$ is the set of projection operators ($P= P^* = P^2$) and $S \leq T$ means that $T- S$ is a non-negative operator. This maximum is reached for $P $ equal to the spectral projection on the the interval $J$, (see e.g. Bercovici and Voiculescu \cite[Lemma 3.2]{MR1254116}).
  
Now let $Q \in \cP$ such that  $\mu_B(J) = \tau(Q)$ and  $Q B Q \leq t Q$. We denote $H$ the range of $Q$ and we consider the projection operator $R$ on $ H \cap \ker ( A- B)$. Observe that
$R A R = R B R \leq t R$. In particular, from \eqref{eq:minmaxJ}, we get
\begin{equation}\label{eq:dimeq}
\tau(R)  = \DIM ( H(Q) \cap \ker (A-B))  \leq \mu_A ( J ) .
\end{equation}
Then, the formula for closed linear subspaces, $U,V$,
$$
\DIM ( U + V) +  \DIM ( U \cap V) = \DIM ( U) + \DIM ( V),
$$
(see  \cite[exercice 8.7.31]{MR1170351}) yields
\begin{eqnarray*}
\DIM ( H \cap \ker (A-B)) & \geq  & \DIM ( H ) + \DIM( \ker (A-B)) - 1 \\
 & \geq & \DIM ( H )  - \RANK(A-B).
\end{eqnarray*}
By definition $\DIM(H) = \mu_B(J)$ and Equation \eqref{eq:dimeq} imply that
$$
\mu_B ( I) - \RANK(A-B) \leq \mu_A(I).
$$
Reversing the role of $A$ and $B$ allows to conclude.
\ep

We can now turn to the proof of Proposition \ref{prop:defspecmeas}. As argued in subsection \ref{subsec:tool2}, there is a natural Von Neumann algebra associated to unimodular measures. We use the canonical way to represent an element $ G \in \cG^ *$ as a rooted graph on the vertex set $V(G) \subset V = \{ o , 1 , 2 , \cdots \}$ with root $o$. We set $H = \ell^2 (V)$  and define $\cB(H)$ as the set of bounded linear operators on $H$. For a fixed  unimodular probability measure $\rho$ in $\cG^*$,  we associate the algebra of bounded operators $\cM  = L^{\infty} ( \cG^*, \cB (H),\rho)$ which commutes with the operators $\lambda_{\sigma} $, defined for all $v \in V$, $\lambda_{\sigma} (e_v) = e_{\sigma(u)}$, where $\sigma : V \to V$ is a bijection.  We endow $\cM$ with the normalized trace
$$\tau (B) = \dE_\rho \langle e_o , B e_o \rangle,$$
where $B = B ( G) \in \cM$ and under, $\dE_\rho$, $G$ has distribution $\rho$.

Remark that $G = ( V (G), E) \in \cG^ *$ can be extended to a graph on $V$ (all vertices in $V \backslash V(G)$ are isolated). Let $n \in \dN$ and consider the adjacency operator $A_n(G)$ of the graph $G_n$ obtained from $G$ by removing all edges adjacent to a vertex of degree larger than $n$ in $G$ : for finitely supported functions of $\psi \in \ell^ 2 (V)$,
$$
A_n (G) \psi (u) = \sum_{v: \{ u , v\} \in E}   \IND ( \deg(u) \leq n )  \IND ( \deg(v) \leq n) \psi(v).
$$
By construction, $A_n$ is bounded : for all $G \in \cG^ *$,
$$
\|A_n (G) \| \leq  n.
$$
Hence $A_n \in \cM$ and the spectral measure  $\mu_{A_n}$ is well-defined (with our notation \eqref{eq:defspecmeas},  if $\rho_n$ is the law of the truncated rooted graph $(G_n(o),o)$, we have $\mu_{A_n} = \mu_{\rho_n}$).

Now since $\rho$ is a probability measure on locally finite graphs,
$$
\dP_\rho \left( \deg (o) > n  \hbox{ or } \, \exists v : \{v ,o\} \in E , \deg(v) > n\right) = \varepsilon (n) \to 0.
$$
Note also that for $B \in \cM$, $\DIM ( \ker (B) ) \geq \dP_\rho ( e_o \in \ker (B))$. We deduce that, for $n,m \in \dN$,
$$
\RANK (A_n - A_{n+m}) \leq 1 - \dP_\rho ( A_n e_o =  A_{n+m} e_o ) \leq \varepsilon (n).
$$
Using Lemma \ref{le:rankineq}, we find that $\mu_{A_n}$ is a Cauchy sequence for the Kolmogorov-Smirnov distance (and hence for the L\'evy distance).  The space $(\cP(\dR), d_{KS})$ is a complete metric space. It follows that $\mu_{A_n}$ converges weakly to some probability measure denoted by $\mu_\rho$ and 
\begin{equation}\label{eq:ksbounded}
d_{KS} ( \mu_{\rho}, \mu_{\rho_n} ) \leq \veps(n). 
\end{equation}
This gives the existence of the spectral measure.

We now prove statement $(i)$ of the proposition and identify $\mu_\rho$ for self-adjoint operators. Note that, $\rho$-a.s., for all $\psi$ with finite support, for all $n$ large enough, $A_n\psi =  A \psi$. Hence, if $\rho$-a.s. $A$ is essentially self-adjoint, this implies the strong resolvent convergence, see e.g. \cite[Theorem VIII.25(a)]{reedsimon}. As a consequence, $\rho$-a.s. $\mu_{A_n}^{e_o}$ converges weakly to $\mu_A ^{e_o}$. Taking expectation, we get  $\mu_\rho = \dE_\rho \mu^ {e_o}_A$.

Let us finally prove statement $(ii)$ of the proposition. Consider a sequence $(\rho_k)$ converging to $\rho$ in the local weak topology. Let $\delta >0$. There exists $n$ such that $\veps(n) < \delta$. By assumption, for all $k \geq k(\delta)$ large enough, $\dP_{\rho_k} \left( \deg (o) > n  \hbox{ or } \, \exists v : \{v ,o\} \in E , \deg(v) > n\right) \leq 2 \delta $. Consider $G'_{k} \in \cG^*$ a random rooted graph with law $\rho_k$ and $G'_{k,n}$ obtained from $G'_k$ by removing all edges adjacent to a vertex of degree larger than $n$ in $G'_k$. We denote by $\rho_{n,k}$ the law of $G'_{k,n}$.  We get from \eqref{eq:ksbounded} that  for all $k\geq k(\delta)$, $d_{KS} (\mu_{\rho_{n,k}}, \mu_{\rho_n})\leq 2 \delta$. Now, from the  Skorokhod's representation theorem one can define a common probability space such that the rooted graphs $G'_k$  converge for the local topology to $G$. In particular, for any compactly supported $\psi \in \ell^ 2 (V)$, for $k$ large enough, $B_{n,k}\psi = A_n \psi$, where $B_{n,k}$ is the adjacency operators of $G'_{k,n}$.  By construction, $B_{n,k}$ and $A_{n}$ are bounded self-adjoint operators. Arguing as in case $(i)$, it implies that  $\mu_{B_{n,k}}^{e_o}$ converges weakly to $\mu_{A_n} ^{e_o}$ as $k\to \infty$. We deduce that 
$$
\lim_{k\to \infty} d_L (  \mu_{B_{n,k}}^{e_o} , \mu_{A_{n}}^{e_o} ) = 0. 
$$
Taking expectation and using the convexity of the distance, we find 
$$
\lim_{k\to \infty} d_L ( \dE \mu_{B_{n,k}}^{e_o} , \dE \mu_{A_{n}}^{e_o} ) = \lim_{k\to \infty} d_L ( \mu_{\rho_{n,k}}, \mu_{\rho_{n}} ) = 0
$$
So finally, $\limsup_k   d_L ( \mu_{\rho_{k}}, \mu_{\rho} ) \leq 3\delta$ and since $\delta > 0$, this concludes the proof of Proposition \ref{prop:defspecmeas}. \ep

\bigskip

\noindent{\bf Acknowledgements.} C.B. and A.S. thank for its hospitality the University of Toronto  where most of this work was done.

\bibliographystyle{abbrv}
\bibliography{mat}

\bigskip
\noindent
 Charles Bordenave \\
 Institut de Math\'ematiques de Toulouse. CNRS and University of Toulouse. \\
118 route de Narbonne. 31062 Toulouse cedex 09.  France. \\
\noindent
{E-mail:} {\tt bordenave@math.univ-toulouse.fr} \\
\noindent
\url{http://www.math.univ-toulouse.fr/~bordenave}

\bigskip
\noindent
Arnab Sen \\
School of Mathematics, University of Minnesota.\\
206 Church Street SE Minneapolis, MN 55455. United States. \\
\noindent
{ E-mail:} {\tt arnab@math.umn.edu} \\
\noindent
\url{ http://math.umn.edu/~arnab}

\bigskip
\noindent
 B\'alint Vir\'ag\\
  Departments of Mathematics and Statistics. University of Toronto.\\
40 St George Street
Toronto, ON, M5S 3G3, Canada.  \\
\noindent
{E-mail:} {\tt
balint@math.toronto.edu.
} \\
\noindent
\url{ http://www.math.toronto.edu/~balint}

\end{document}